\documentclass[12pt,a4paper, oneside, reqno]{amsart}
\usepackage[foot]{amsaddr} 
\usepackage{mathtools}
\usepackage{amsmath, nicefrac, amsthm, verbatim, amsfonts, mathtools, amssymb, upgreek, xcolor, bbm}
\usepackage{amsmath, nicefrac, amsthm, verbatim, amsfonts, amssymb, xcolor, bbm}
\usepackage{graphics, xspace, enumerate}
\usepackage{enumitem}
\usepackage{stix}
\usepackage[top=30mm,bottom=30mm,left=25mm,right=25mm]{geometry}

\usepackage[colorinlistoftodos,  textsize=tiny]{todonotes} 					 
\presetkeys{todonotes}{fancyline, color=blue!30}{}

\usepackage[bb=boondox]{mathalfa}
\usepackage{graphicx}
\usepackage[colorlinks=true,citecolor=red,urlcolor=blue,linkcolor=red,bookmarksopen=true,unicode=true,pdffitwindow=false]{hyperref}
\usepackage[english]{babel}
\usepackage[languagenames,fixlanguage]{babelbib}
\hypersetup{pdfauthor={}}
\hypersetup{pdftitle={A note on the uniform ergodicity of diffusion processes}}

\theoremstyle{plain}
\newtheorem{theorem}{Theorem}[section]

\newtheorem{proposition}[theorem]{Proposition}
\theoremstyle{definition}

\newtheorem{example}[theorem]{Example}

\newcommand {\Prob} {\ensuremath{\mathbb{P}}}
\newcommand {\R} {\ensuremath{\mathbb{R}}}
\newcommand {\ZZ} {\ensuremath{\mathbb{Z}}}
\newcommand {\N} {\ensuremath{\mathbb{N}}}

\newcommand{\df}{\coloneqq}

\newcommand{\E}{\mathrm{e}}

\newcommand{\X}{\mathrm{X}}

\newcommand{\calC}{\mathcal{C}}

\newcommand{\D}{\mathrm{d}}

\numberwithin{equation}{section}

\title{A note on the uniform ergodicity of diffusion processes}

\author[N.\ Sandri\'{c}]{Nikola Sandri\'{c}}
\address[Nikola\ Sandri\'{c}]{Department of Mathematics\\University of Zagreb\\ Zagreb\\Croatia}
\email{nsandric@math.hr}

\makeatletter
\@namedef{subjclassname@2020}{%
	\textup{2020} Mathematics Subject Classification}
\makeatother


\begin{document}
\allowdisplaybreaks[4]

\begin{abstract}
 In this note, we discuss the uniform ergodicity of a diffusion process given by an Itô stochastic differential equation. We present an integral condition in terms of the drift and diffusion coefficients that ensures the uniform ergodicity of the corresponding transition kernel with respect to the total variation distance. Applications of the obtained results to a class of subordinate diffusion processes are also presented.
\end{abstract}

\maketitle

\section{Introduction}

One of the classical directions in the analysis of Markov processes centers around their ergodicity properties. In this note, we study the uniform ergodicity of a diffusion process given by  
\begin{equation}\begin{aligned} \label{eq1}\D X(x,t)&=b(X(x,t)) +\sigma(X(x,t))\D B(t)\\X(x,0)&=x\in\R^{d}\end{aligned}\end{equation}
with respect to the total variation distance.
Here, $\{B(t)\}_{t\ge0}$ stands for a standard $n$-dimensional Brownian motion (defined on a stochastic basis $(\Omega,\mathcal{F},\{\mathcal{F}(t)\}_{t\ge0},\mathbb{P})$ satisfying the usual conditions), and the coefficients $b:\R^d\to\R^d$ and $\sigma:\R^d\to\R^{d\times n}$ satisfy the following conditions: 

\medskip

\begin{description}
	\item[(A1)] For any $r>0$,
	$$\sup_{x\in B_r(0)}(\rvert b(x)\rvert+\lVert\sigma(x)\lVert_{{\rm HS}})<\infty.$$
	\item[(A2)] For any $r>0$, there exists $\Gamma_r>0$ such that for all $x,y\in B_r(0)$,  $$2\langle x-y,b(x)-b(y)\rangle+\lVert\sigma(x)-\sigma(y)\lVert_{{\rm HS}}^2\le \Gamma_r|x-y|^{2}.$$
	\item[(A3)] There exists $\Gamma>0$ such that for all $x\in\R^d$,  $$2\langle x,b(x)\rangle+\lVert\sigma(x)\lVert_{{\rm HS}}^{2}\le\Gamma(1+|x|^2).$$
	\item[(A4)] There exist  $x_0\in\R^d$ and $r_0,\alpha,D,\delta>0$,  such that for all   $x,y\in B_{r_0}(x_0)$,
	$$|b(x)-b(y)|+\lVert \sigma(x)\sigma(x)^T-\sigma(y)\sigma(y)^T\rVert_{{\rm HS}}\le D|x-y|^{\alpha}\quad\textrm{and}\quad \langle \sigma(x)^Ty,\sigma(x)^Ty\rangle\ge \delta|y|^{2}.$$
		\item[(A5)] For $x_0\in\R^d$ and $r_0>0$ from (A4),  $\sigma(x)\sigma(x)^T$ is positive definite on $B^c_{r_0}(x_0)$.
\end{description}

\medskip

\noindent Here,  $B_r(x)$ denotes the open ball with radius $r>0$ around $x\in\R^d$, and  $\lVert M\lVert_{{\rm HS}}^2:={\rm Tr}\,MM^T$ is the  Hilbert-Schmidt norm of a real  matrix $M.$

\subsection{Structural properties of the model} \label{ss1}
Under (A1)-(A3), for any $x\in\R^d$, the stochastic differential equation (SDE) in \eqref{eq1} admits a unique strong non-explosive solution $\{X(x,t)\}_{t\ge0}$ which is a strong Markov process with continuous sample paths and  transition kernel $p(t,x,\D y)=\Prob(X(x,t)\in \D y)$, $t\ge0$, $x\in\R^d$. 
See  \cite[Theorems 5.4.1, 5.4.5 and 5.4.6]{Durrett-Book-1996} and \cite[Theorem 3.1.1]{Prevot-Rockner-Book-2007} for more details. In the context of Markov processes, it is natural that the underlying probability measure depends on the initial distribution of the process. Using standard arguments (Kolmogorov extension theorem), it is well known that for each $x\in\R^d$ the above-defined transition kernel defines a unique probability measure $\mathbb{P}^x$ on the canonical (sample-path)  space such that the projection process, denoted by $\{X(t)\}_{t\ge0}$, is a strong Markov process (with respect to the completion of the corresponding natural filtration), it has continuous  sample paths, 
and  the same finite-dimensional distributions (with respect to $\mathbb{P}^x$) as $\{X(x,t)\}_{t\ge0}$ (with respect to $\mathbb{P}$). Since we are interested in distributional properties of the solution to   \eqref{eq1}, we focus on $\{X(t)\}_{t\ge0}$ rather than  $\{X(x,t)\}_{t\ge0}$ in the sequel. Furthermore, according to \cite[Lemma 2.5]{Majka-2020}, for any $t>0$ and continuous and bounded $f:\R^d\to\R$ the mapping $$x\mapsto\int_{\R^d}f(y)p(t,x,\D y)$$ is also continuous and bounded, i.e., $\{X(t)\}_{t\ge0}$ satisfies the so-called $\calC_b$-Feller property.   
In particular, this automatically implies that $\{X(t)\}_{t\ge0}$ is a strong Markov process with respect to the right-continuous and completed version of the  natural filtration.
Finally, in \cite[Theorem V.21.1]{Rogers-Williams-Book-II-2000} it is shown that
the class $\mathcal{C}^2(\R^d)$ is contained in the domain of the extended generator $\mathcal{G}$ of $\{X(t)\}_{t\ge0}$, and $\mathcal{G}$ restricted to this class has the following form:
\begin{equation}\label{gen}\mathcal{G}f(x)=\langle b(x),\nabla f(x)\rangle+\frac{1}{2}{\rm Tr}\,\sigma(x)\sigma(x)^T\nabla^2f(x).\end{equation}

\subsection{Main result} \label{ss2} Before stating the main result of this article, we first recall definitions of ergodicity. 

\medskip

\begin{enumerate}
\item [(i)]  A
(not necessarily finite) measure $\uppi(\D x)$ on $\mathcal{B}(\R^d)$ is called invariant for
$\{X(t)\}_{t\ge0}$ if
$$\int_{\R^d}p(t,x,\D y)\uppi(\D x)
= \uppi(\D y),\qquad t>0.$$
	\item [(ii)] The process $\{X(t)\}_{t\ge0}$  is said to be ergodic
	if it possesses an invariant probability 
	measure $\uppi(\D x)$  such that 
	\begin{equation*}
	\lim_{t\to\infty}\lVert p(t,x,\D {y})
	-\uppi(\D {y})\rVert_{{\rm TV}} =0,\qquad x \in \R^d.
	\end{equation*}
		\item [(iii)] The process $\{X(t)\}_{t\ge0}$  is said to be uniformly ergodic 
	if it possesses an invariant probability 
	measure $\uppi(\D x)$  such that 
	\begin{equation*}
	\lim_{t\to\infty}\sup_{x\in \R^d}\lVert p(t,x,\D {y})
	-\uppi(\D {y})\rVert_{{\rm TV}} =0.
	\end{equation*}
\end{enumerate}

\medskip

\noindent Here, $\lVert\upmu\rVert_{{\rm TV}}\df\sup_{B\in\mathcal{B}(\R^d)}|\upmu(B)|$ is the total variation norm  of a signed measure $\upmu(\D x)$ (on $\mathcal{B}(\R^d)$). Observe that: 

\medskip

\begin{enumerate}
	\item [(a)] (Uniform) ergodicity immediately implies the uniqueness of $\uppi(\D x)$.
	
	\item[(b)] The rate of convergence in (ii) depends on structural properties of the process (see, e.g., \cite{Douc-Fort-Guilin-2009} and \cite{Down-Meyn-Tweedie-1995}). 
	
	\item[(c)] The  rate of convergence  in (iii)  is necessarily exponential. Namely,  from (iii) we have that for any $t_0>0$,
	\begin{equation*}
	\lim_{n\to\infty}\sup_{x\in \R^d}\lVert p(nt_0,x,\D {y})
	-\uppi(\D {y})\rVert_{{\rm TV}} =0,
	\end{equation*} i.e., the Markov chain $\{X(nt_0)\}_{n\in\ZZ_+}$ is uniformly ergodic. Consequently, from \cite[Theorem 16.0.2]{Meyn-Tweedie-Book-2009} it follows that there exist $B,\beta>0$ such that 
	\begin{equation*}
	\lVert p(nt_0,x,\D {y})
	-\uppi(\D {y})\rVert_{{\rm TV}} \le B\E^{-\beta  n t_0},\qquad  x \in \R^d, \quad n\in\ZZ_+.\end{equation*} Since every $t\ge0$ can be clearly represented as $t=nt_0+s$ for some $n\in\ZZ_+$ and $s\in[0,t_0)$, we then conclude \begin{align*}\lVert p(t,x,\D {y})
	-\uppi(\D {y})\rVert_{{\rm TV}}&=\lVert p(nt_0+s,x,\D {y})
	-\uppi(\D {y})\rVert_{{\rm TV}}\\&\le\lVert p(nt_0,x,\D {y})
	-\uppi(\D {y})\rVert_{{\rm TV}}  \\&\le B\E^{-\beta  n t_0}
	\\&=B\E^{\beta s}\E^{-\beta  t}\\&\le B\E^{\beta t_0}\E^{-\beta  t},\qquad x \in \R^d, \quad t>0.\end{align*}
\end{enumerate}

\medskip

\noindent We are now in position to state the main result of this article.
Let $x_0\in\R^d$ and $r_0>0$ be as in (A4) and (A5). Set
\begin{align*}
A(x)&\df\frac{1}{2}{\rm Tr}\,\sigma(x)\sigma(x)^{T},\qquad x\in\R^{d},\\
B_{x_0}(x)&\df\langle x-x_0,b(x)\rangle,\qquad x\in\R^{d},\\
C_{x_0}(x)&\df\frac{\langle \sigma(x)^T(x-x_0),\sigma(x)^{T}(x-x_0)\rangle}{|x-x_0|^{2}},\qquad x\in\R^{d}\setminus\{x_0\},\\
\gamma_{x_0}(r)&\df\inf_{|x-x_0|=r}C_{x_0}(x),\qquad r\geq r_0,\\
\iota_{x_0}(r)&\df\sup_{|x-x_0|=r}\frac{2A(x)-C_{x_0}(x)+2B_{x_0}(x)}{C_{x_0}(x)},\qquad r\geq r_0,\\
I_{x_0}(r)&\df\int_{r_0}^{r}\frac{\iota_{x_0}(s)}{s}ds,\qquad r\geq r_0.
\end{align*}

\medskip

\begin{theorem}\label{tm:TV} Assume $\textbf{(A1)-(A5)}$, and  
	\begin{equation}\label{eq2}\Lambda\df\int_{r_0}^{\infty}\E^{-I_{x_0}(u)}\int_u^\infty  \frac{\E^{I_{x_0}(v)}}{\gamma_{x_0}(v)}\D v \,\D u<\infty\end{equation} (observe that (A5) ensures that $\Lambda$ is well defined).
	Then, $\{X(t)\}_{t\ge0}$ is uniformly ergodic.
\end{theorem}

\medskip

\noindent The   proof of Theorem \ref{tm:TV} is based on the Foster-Lyapunov method for uniform ergodicity of Markov processes, developed in \cite{Down-Meyn-Tweedie-1995}.
The method consists of identifying an appropriate recurrent (petite) set $C\in\mathcal{B}(\R^d)$ and constructing an appropriate  bounded function $L:\R^d\to[1,\infty)$ (the so-called Lyapunov  function) contained in the domain of  $\mathcal{G}$  (see \cite[Section 4]{Down-Meyn-Tweedie-1995} for details), such that the Lyapunov equation \begin{equation}
\label{eq:lyap}
\mathcal{G}L(x)\le-c_1L(x)+c_2\mathbb{1}_C(x),\qquad x\in\R^d,\end{equation} holds for some $c_1>0$ and $c_2\in\R$ (see \cite[Theorem 5.2]{Down-Meyn-Tweedie-1995}). 
Recall that    a set $C\in\mathcal{B}(\R^d)$ is said to be   petite if it satisfies the Harris-type minorization condition: there exist a probability   measure $\upeta_\mathcal{C}(\D t)$ on $\mathcal{B}((0,\infty))$ (the standard Borel $\sigma$-algebra on $(0,\infty)$) and a non-trivial measure $\upnu_C(\D x)$ on $\mathcal{B}(\R^d)$, such that $\int_0^\infty p(t,x,B)\upeta_C(\D t)\ge\upnu_C(B)$ for all $x\in C$ and $B\in\mathcal{B}(\R^d)$. Intuitively,  petite sets serve the role of singletons for Markov processes on non-discrete state spaces (see \cite[Chapter 5]{Meyn-Tweedie-Book-2009} and \cite[Section 4]{Meyn-Tweedie-AdvAP-II-1993}  for details).
However, the Lyapunov equation does not immediately imply ergodicity of $\{X(t)\}_{t\ge0}$. As in the discrete setting (see, e.g., \cite[Chapter 13]{Meyn-Tweedie-Book-2009}),  the process  has to satisfy additional structural properties (irreducibility and aperiodicity). Note that so far, we have not  used assumptions (A4) and (A5). In Proposition \ref{P1} we show that (A4) and (A5), together with (A1)-(A3), imply that  $\{X(t)\}_{t\ge0}$ is 

\medskip

\begin{enumerate}
	\item [(i)]	open-set irreducible, i.e., there exists a $\sigma$-finite measure $\uppsi(\D x)$ on
	$\mathcal{B}(\R^d)$ such that $${\rm supp}\,\uppsi\df\{x\in\R^d: \uppsi(O)>0\ \text{for every open neighborhood}\ O\ \text{of}\ x\}$$ has a non-empty interior, and whenever $\uppsi(B)>0$ we have
	$\int_0^{\infty}p(t,x,B)\D t>0$ for all $x\in\R^d$.
	
		\item [(ii)] aperiodic, i.e.,  
	there exist $t_0>0$ and a $\sigma$-finite measure $\upphi(\D x)$ on
	$\mathcal{B}(\R^d)$, such that $\upphi(B)>0$ implies
	$\sum_{n=0}^{\infty} p(nt_0,x,B) >0$ for all $x\in\R^d$.
\end{enumerate}

\medskip

\noindent Let us also remark that,  according to \cite[Propositio 4.1]{Meyn-Tweedie-AdvAP-II-1993} and  \cite[Theorems 5.1 and 7.1]{Tweedie-1994}, the $\mathcal{C}_b$-Feller property and open-set irreducibility of $\{X(t)\}_{t\ge0}$ ensure that the state space (in this case, $(\R^d,\mathcal{B}(\R^d)$) can be covered by a countable union of petite sets, and  that  every compact set is  petite.

To the best of our knowledge, Theorem \ref{tm:TV}  presents the first explicit conditions for the uniform ergodicity of diffusion processes based on the Lyapunov method. As previously noted, general conditions using the Lyapunov method for the uniform ergodicity of general Markov processes are derived in \cite{Down-Meyn-Tweedie-1995}. The only related results on this topic available in the literature (under the terminology of strong ergodicity) can be found in \cite[Section ]{Mao-2002}, \cite[Section 3]{Mao-2006} and \cite[Section 3]{Shao-Xi-2013}. The main difference is that the results in these works are based on the coupling method.   Therefore, the presented conditions are of the incremental type, which  in certain cases they can be  less practical, i.e. more difficult to verify (compare Examples \ref{ex1}, \ref{ex2} and \ref{ex3} with the examples in \cite{Mao-2002}, \cite{Mao-2006} and \cite{Shao-Xi-2013}).

\subsection{Literature review} \label{ss4}
Our work relates to active research on the ergodicity properties of Markov processes and the extensive literature on SDEs. 
In \cite{Arapostathis-Borkar-Ghosh_Book-2012}, \cite{Bhattacharya-1978}, \cite{Kulik-Book-2015}, \cite{Kulik-Book-2018}, \cite{Lazic-Sandric-2021}, \cite{Lazic-Sandric-2022}, \cite{Stramer-Tweedie-1997}, and \cite{Veretennikov-1997},   ergodicity properties of diffusion processes with respect to the total variation distance  are established using the Lyapunov(-type) method.
 In this article, building  on the ideas from \cite{Bhattacharya-1978} (see also \cite[Chapter 9]{Friedman-Book-1975} and \cite[Supplement]{Khasminskii-1960}), we complement these results by deriving  conditions 
for  the uniform ergodicity  of  these processes. Additionally, we apply these results to discuss the uniform ergodicity of a class of Markov jump processes obtained through the Bochner's subordination of diffusion processes. These results are related to 
\cite{Albeverio-Brzezniak-Wu-2010}, \cite{Arapostathis-Pang-Sandric-2019}, \cite{Deng-Schilling-Song-2017, Deng-Schilling-Song-Erratum-2018} \cite{Douc-Fort-Guilin-2009}, \cite{Down-Meyn-Tweedie-1995},
\cite{Fort-Roberts-2005}, \cite{Kevei-2018},  \cite{Kulik-2009},   \cite{Masuda-2007, Masuda-Erratum-2009}, \cite{Meyn-Tweedie-AdvAP-II-1993}, \cite{Meyn-Tweedie-AdvAP-III-1993}, \cite{Sandric-ESAIM-2016},  \cite{FYWang-2011}, \cite{Wang-2008}, \cite{Wang-2011}   and \cite{Wee-1999},  
where the ergodicity properties  of  general 
Markov processes are again established using the  Foster-Lyapunov method.  Analogous results have also been obtained in the discrete-time setting, see   \cite{Durmus-Fort-Moulines-Soulier-2004}, \cite{Douc-Moulines-Priouret-Soulier_Book-2018},   \cite{Fort-Moulines-2003}, \cite{Kulik-Book-2015}, \cite{Kulik-Book-2018}, \cite{Meyn-Tweedie-Book-2009}, \cite{Veretennikov-1997}, \cite{Veretennikov-2000}, \cite{Tuominen-Tweedie-1994} and the references therein.

The studies on ergodicity properties with respect to the total variation distance assume that 
the Markov processes are irreducible and aperiodic. This is satisfied if the process does not exhibit singular behavior in its motion, i.e., its diffusion part is non-singular and/or its jump part shows sufficient jump activity. Together with the Foster-Lyapunov condition, which ensures controllability of modulated moment of return times
to a petite set,  
these conditions lead to the ergodic properties stated.
In this article, we also discuss and provide conditions for the irreducibility and aperiodicity of diffusion processes. There is a vast literature on these, and related questions such as the strong
Feller property and heat kernel estimates of Markov processes. 
In particular, we refer the readers to 
\cite{Arisawa-2009}, \cite{Chen-Chen-Wang-2020}, \cite{Chen-Hu-Xie-Zhang-2017}, \cite{Chen-Zhang-2016},  \cite{Chen-Zhang-2018}, \cite{Grzywny-Szczypkowski-2019}, \cite{Kim-Lee-2019}, \cite{Kim-Song-Vondracek-2018}, \cite{Knopova-Schilling-2012}, \cite{Knopova-Schilling-2013},
	\cite{Kolokoltsov-2000}, \cite{Kolokoltsov-Book-2011}, \cite{Kwon-Lee-1999}, \cite{Pang-Sandric-2016},
	\cite{Sandric-TAMS-2016}, and \cite{Stroock-1975}
for the case of a class of Markov L\'evy-type processes with bounded coefficients,
and to
\cite{Arapostathis-Pang-Sandric-2019}, \cite{Bass-Cranston-1986}, \cite{Ishikawa-2001},
	\cite{Knopova-Kulik-2014}, \cite{Lazic-Sandric-2021}, \cite{Masuda-2007}, \cite{Masuda-Erratum-2009}, 
	\cite{Picard-1996, Picard-Erratum-2010}, \cite{Sandric-Valentic-Wang-2021}, \cite{Stramer-Tweedie-1997}, and \cite{Xi-Zhu-2019}
for the case of a class of It\^{o} processes.

At the end, we  remark that for Markov processes that do not converge in total variation (e.g., due to a lack of irreducibility and/or aperiodicity), 
ergodic properties under  Wasserstein distances are studied, as these processes may converge weakly under certain conditions. For instance, see \cite{Bolley-Gentil-Guillin-2012}, 
\cite{Butkovsky-2014}, \cite{Durmus-Fort-Moulines-2016}, \cite{Eberle-2011}, \cite{Eberle-2015}, \cite{Hairer-Mattingly-Scheutzow-2011}, \cite{Luo-Wang-2016}, \cite{Majka-2017}, \cite{Renesse-Sturm-2005} and \cite{Wang-2016}. 

\section{Proof of Theorem  \ref{tm:TV}}\label{s2}
In this section, we prove Theorem  \ref{tm:TV}. First, we prove that assumptions (A4) and (A5), together with (A1)-(A3), imply the open-set irreducibility and aperiodicity of  $\{X(t)\}_{t\ge0}$.

\begin{proposition}\label{P1} Under assumptions (A1)-(A5), the process 
	$\{X(t)\}_{t\ge0}$ is open-set irreducible and aperiodic.
	\end{proposition}
\begin{proof}
	We first show that (A5), together with (A1)-(A3), implies that $$\mathbb{P}^x(\uptau_{B_{r_0}(x_0)}<\infty)>0, \qquad  x\in\R^d,$$ where $x_0\in\R^d$ and $r_0>0$ are given in (A5), and $ \uptau_{B_{r_0}(x_0)}$ is defied as $\uptau_{B_{r_0}(x_0)}\df\inf\{t\ge0:X(t)\in B_{r_0}(x_0)\}$.
		Let $0<\varepsilon<r_0$, and define  $$\bar{L}(r)\df\int_{r_0-\varepsilon}^{r}\E^{-I_{x_0}(u)}\D u,\qquad r\geq r_0-\varepsilon.$$ 
	For $r> r_0-\varepsilon$ we then have
	$$
	\bar{L}'(r)=\E^{-I_{x_0}(r)}>0\qquad \text{and}\qquad 
	\bar{L}''(r)=-\frac{\bar{L}'(r)}{r}\iota_{x_0}(r).
	$$		
	Further, let $L:\R^{d}\to[0,\infty)$, $L\in \mathcal{C}^2(\R^d)$, be such that $L(x)=\bar{L}(|x-x_0|)$ for $x\in B^c_{r_0}(x_0)$. For $x\in B^c_{r_0}(x_0)$, we then have
	\begin{align*}
	\mathcal{G}L(x)&=\frac{1}{2}C_{x_0}(x)\bar{L}''(|x-x_0|)+\frac{\bar{L}'(|x-x_0|)}{2|x-x_0|}(2A(x)-C_{x_0}(x)+2B_{x_0}(x))\\
	&= \frac{\bar{L}'(|x-x_0|)}{2|x-x_0|}\left(2A(x)-C_{x_0}(x)+2B_{x_0}(x)-C_{x_0}(x)\iota_{x_0}(|x-x_0|)\right),
	\end{align*}
	and thus $$	\mathcal{G}L(x)\le0.$$
	Further, for $n\in\N$,  define 
	$\tau_n\df\uptau_{B^c_n(x_0)}.$ Clearly, $\uptau_n$, $n\in\N$, are stopping times such that (due to non-explosivity of $\{X(t)\}_{t\ge0}$) we have $\uptau_n\to\infty$, $\Prob^x$-a.s.,  as $n\to\infty$ for all $x\in\R^d$.
	Hence, the processes (recall that $\mathcal{G}$ is the extended generator of $\{X(t)\}_{t\ge0}$) $$L(X(t\wedge\tau_n))-L(X(0))-\int_0^{t\wedge\uptau_n}\mathcal{G}L(X(s))\D s,\qquad t\geq0\quad n\in\N,$$ are $\Prob^x$-martingales. For $x\in B^c_{r_0}(x_0)$, it then holds that 
	\begin{align*}
	\mathbb{E}^x[L(X(t\wedge\tau_n\wedge\tau_{B_{r_0}(x_0)}))]-\mathbb{E}^x[L(X(0))]&=\mathbb{E}^x\int_0^{t\wedge\tau_n\wedge\tau_{B_{r_0}(x_0)}}\mathcal{G}L(X(s))\D s\leq 0,
	\end{align*}
	i.e.,
	$$\mathbb{E}^x[\bar{L}(|X(t\wedge\tau_n\wedge\tau_{B_{r_0}(x_0)})-x_0|)]\le \bar{L}(|x-x_0|).$$
	In particular,
	$$\mathbb{E}^x[\bar{L}(|X(t\wedge\tau_n)-x_0|)\mathbb{1}_{\{\tau_{B_{r_0}(x_0)}>\tau_n\}}]\le \bar{L}(|x-x_0|),\qquad x\in B^c_{r_0}(x_0).$$ By letting $t\to\infty$, Fatou's lemma implies 
	$$\bar{L}(n)\mathbb{P}^x(\tau_{B_{r_0}(x_0)}>\tau_n)\le \bar{L}(|x-x_0|),\qquad x\in B^c_{r_0}(x_0).$$ Hence, by letting $n\to\infty$, we conclude that
	$$ \mathbb{P}^x(\tau_{B_{r_0}(x_0)}=\infty)\le \frac{\bar{L}(|x-x_0|)}{\bar L(\infty)}<1,\qquad x\in B^c_{r_0}(x_0),$$ i.e., $\mathbb{P}^x(\tau_{B_{r_0}(x_0)}<\infty)>0$ for all $x\in\R^d.$

	We next show that $\{X(t)\}_{t\ge0}$ is open-set irreducible. According  to \cite[Theorems 7.3.6 and 7.3.7]{Durrett-Book-1996}, there exists a strictly positive function  $q(t,x,y)$ on $(0,\infty)\times \bar{B}_{r_0}(x_0)\times\bar{B}_{r_0}(x_0)$, jointly continuous in $t$, $x$ and $y$, and twice continuously differentiable in $x\in B_{r_0}(x_0)$, satisfying  \begin{equation*} \mathbb{E}^x[f(X(t)) \mathbb{1}_{\{\uptau_{\bar{B}^c_{r_0}(x_0)}>t\}}]=\int_{B_{r_0}(x_0)}q(t,x,y)f(y)\D y,\end{equation*}   for any $t>0$, $x\in B_{r_0}(x_0)$ and any continuous, bounded  function $f:\R^d\to\R.$ By employing the dominated convergence theorem, the above relation holds also for $\mathbb{1}_O$, where $O\subseteq B_{r_0}(x_0)$ is an open set. Denote by $\mathcal{D}$ the class of all  $B\in\mathcal{B}(B_{r_0}(x_0))$ (the Borel $\sigma$-algebra on $B_{r_0}(x_0)$) such that \begin{equation*} \mathbb{P}^x(X_t\in B,\ \uptau_{\bar{B}^c_{r_0}(x_0)}>t)=\int_{B}q(t,x,y) \D y,\qquad t>0,\quad x\in B_{r_0}(x_0).\end{equation*} Clearly, $\mathcal{D}$ contains the $\pi$-system of open rectangles in $B_{r_0}(x_0)$, and forms a $\lambda$-system. By  Dynkin's $\pi$-$\lambda$ theorem, we conclude that $\mathcal{D}=\mathcal{B}(B_{r_0}(x_0)).$ Consequently, for any $t>0$, $x\in B_{r_0}(x_0)$ and $B\in\mathcal{B}(\R^d)$, we have  \begin{equation*} p(t,x,B)\ge \int_{B\cap B_{r_0}(x_0)}q(t,x,y)\D y.\end{equation*} Define  $\uppsi(\cdot)\df\uplambda(\cdot\cap B_{r_0}(x_0))$, where $\uplambda(\D x)$ denotes the Lebesgue measure on $\R^d$. Clearly, $\uppsi(\D x)$ is a $\sigma$-finite measure whose support has a non-empty interior. It remains to prove that  $$\int_0^{\infty}p(t,x,B)\D t>0,\qquad x\in\R^d,$$ whenever $\uppsi(B)>0$, $B\in \mathcal{B}(\R^d)$.
	For $x\in B_{r_0}(x_0)$ the implication holds trivially.
	Let  $x\in B^c_{r_0}(x_0)$  and $B\in\mathcal{B}(\R^d)$, $\uppsi(B)>0$, be arbitrary. For all $s>0$ we have
	\begin{align*}
	\int_0^\infty p(t,x,B) \D t & \ge \int_s^\infty  p(t,x,B)\D t \\&= \int_s^\infty \int_{\R^d} p(t-s,x,\D y)p(s,y,B) \D t \\
	&\geq \int_s^\infty \int_{B_{r_0}(x_0)} p(t-s,x,\D y)p(s,y,B) \D t \\
	&= \int_{B_{r_0}(x_0)}p(s,y,B) \int_s^\infty p(t-s,x,\D y) \D t .
	\end{align*}
	The assertion follows from the fact that $p(s,y,B)>0$ for $y\in B_{r_0}(x_0)$, and 
	\begin{equation*} \int_s^\infty p(t-s,x,B_{r_0}(x_0))\D t
	=\int_0^\infty p(t,x,B_{r_0}(x_0)) \D t
	=\mathbb{E}^x\left[\int_0^\infty \mathbb{1}_{\{X(t)\in B_{r_0}(x_0)\}}\D t\right]>0,
	\end{equation*}
	since $\{X(t)\}_{t\ge0}$ has continuous sample paths, $B_{r_0}(x_0)$ is an open set, and $\mathbb{P}^x(\uptau_{B_{r_0}(x_0)}<\infty)>0$ for every $x\in\R^d$.
	
	Finally, we show that $\{X(t)\}_{t\ge0}$ is aperiodic.  Specifically, we show that $$ \sum_{n=1}^\infty p(n,x,B)>0,\qquad x\in\R^d,$$ whenever $\uppsi(B)>0,$  where $B\in \mathcal{B}(\R^d)$.	For $x\in B_{r_0}(x_0)$, this relation clearly holds. For $x\in B^c_{r_0}(x_0)$ and $B\in\mathcal{B}(\R^d)$, with $\uppsi(B)>0$, we have \begin{equation*}
	\sum_{n=1}^\infty p(n,x,B)\ge\int_{B_{r_0}(x_0)}\sum_{n=1}^\infty p(n-t,x,\D y)p(t,y,B),\qquad t\in(0,1).
	\end{equation*} Since $p(t,y,B)>0$ for $y\in B_{r_0}(x_0)$, it suffices to show that \begin{equation*}\sum_{n=1}^\infty p(n-t,x,B_{r_0}(x_0))\ge\mathbb{P}^x\left(\bigcup_{n=1}^\infty\{X(n-t)\in B_{r_0}(x_0)\}\right)>0\end{equation*} for some $t\in(0,1)$. Now, assume by contradiction that this is not the case, i.e., \begin{equation*} \mathbb{P}^x\left(\bigcup_{n=1}^\infty\{X(n-t)\in B_{r_0}(x_0)\}\right)=0,\qquad t\in(0,1)\,.\end{equation*} This implies that \begin{equation*}\mathbb{P}^x\left(\bigcup_{q\in\mathbb{Q}_+\setminus\ZZ_+}\{X(q)\in B_{r_0}(x_0)\}\right)=0,\end{equation*} which is impossible, since $\{X(t)\}_{t\ge0}$ has continuous sample paths, $B_{r_0}(x_0)$ is an open set and $\mathbb{P}^x(\uptau_{B_{r_0}(x_0)}<\infty)>0$ for every $x\in\R^d$. Therefore, the original assumption must be false, and we conclude that $$ \sum_{n=1}^\infty p(n,x,B)>0$$ for all $x\in\R^d$ and $B\in \mathcal{B}(\R^d)$ with $\uppsi(B)>0,$  proving aperiodicity.
	\end{proof}

We are now in position to prove Theorem \ref{tm:TV}. To do so, we  apply \cite[Theorem 5.2]{Down-Meyn-Tweedie-1995}. 
However, note that the notion of aperiodicity used in \cite[Theorem 5.2]{Down-Meyn-Tweedie-1995} differs slightly from the one used in this article. Therefore, we need to show that the  aperiodicity defined in this article implies the notion used in \cite{Down-Meyn-Tweedie-1995}. 
Indeed, by
\cite[Proposition 6.1]{Meyn-Tweedie-AdvAP-II-1993}, \cite[Theorem 4.2]{Meyn-Tweedie-AdvAP-III-1993}, along with the 
aperiodicity of $\{X(t)\}_{t\ge0}$, there exists  a petite set $C\in\mathcal{B}(\R^d)$, a time 	 $T>0$, and a non-trivial measure $\upnu_C(\D x)$ on $\mathcal{B}(\R^d)$, such that $\upnu_C(C)>0$ and $$p(t,x,B)\ge\upnu_C(B),\qquad x\in C,\quad t\ge T,\quad B\in\mathcal{B}(\R^d).$$ In particular, we have
$$p(t,x,C)>0,\qquad x\in C,\quad t\ge T,$$	
which aligns exactly with the definition of aperiodicity used in \cite{Down-Meyn-Tweedie-1995}.

\begin{proof}[Proof of Theorem \ref{tm:TV}]
Define
$$\bar{L}(r)\df\int_{r_0}^{r}\E^{-I_{x_0}(u)}\int_u^{\infty}\frac{\E^{I_{x_0}(v)}}{\gamma_{x_0}(v)}\D v\,\D u,\qquad r\geq r_0.$$  Clearly, for  $r\ge r_0$ it holds that $\bar{L}(r)\le\Lambda$,
 \begin{align*}\bar{L}'(r)=\E^{-I_{x_0}(r)}\int_r^{\infty}\frac{\E^{I_{x_0}(u)}}{\gamma_{x_0}(u)}\D u\qquad\text{and}\qquad
\bar{L}''(r)=-\iota_{x_0}(r)\frac{\bar{L}'(r)}{r}-\frac{1}{\gamma_{x_0}(r)}.
\end{align*}
Further, fix $r_1>r_0$  and let $L:\R^{d}\to[0,\infty)$, $L\in \mathcal{C}^2(\R^d)$, be such that $L(x)=\bar{L}(|x-x_0|)+1$ for $x\in B^c_{r_1}(x_0).$ Now, for $x\in B^c_{r_1}(x_0)$, we have (recall the representation of $\mathcal{G}$ given in \eqref{gen})
\begin{align*}
\mathcal{G}L(x)&=\frac{1}{2}C_{x_0}(x)\bar{L}''(|x-x_0|)+\frac{\bar{L}'(|x-x_0|)}{2|x-x_0|}(2A(x)-C_{x_0}(x)+2B_{x_0}(x))\\
&\le \frac{\bar{L}'(|x-x_0|)}{2|x-x_0|}(2A(x)-C_{x_0}(x)+2B_{x_0}(x)-C_{x_0}(x)\iota_{x_0}(x))-\frac{C_{x_0}(x)}{2\gamma_{x_0}(|x-x_0|)}
\\
&\le-\frac{C_{x_0}(x)}{2\gamma_{x_0}(|x-x_0|)}\\
&\le -\frac{1}{2}\\
&\le-\frac{1}{2(\Lambda +1)}L(x).
\end{align*}
 Thus, we have obtained the relation in \eqref{eq:lyap} with $c_1=1/(2(\Lambda +1))$, $C=\bar B_{r_1}(x_0)$ and $c_2=\sup_{x\in C}|\mathcal{G}L(x)|$, which proves the assertion of the theorem.
 \end{proof}

 \section{Examples and Applications of  Theorem  \ref{tm:TV}}\label{s3}
 In this section, we provide several examples of diffusion processes that satisfy the assumptions of  Theorem  \ref{tm:TV}, along with an application of Theorem  \ref{tm:TV} to a class of Markov processes with jumps.
We begin with a simple example of a diffusion process with polynomial drift.
 
 \begin{example}\label{ex1} Let $\sigma(x)=I_d$ (the $d\times d$ identity matrix) and let $b(x)=-Kx|x|^{\kappa-1}$ with $K>0$ and $\kappa>1$. It is straightforward to verify that $b(x)$ and $\sigma(x)$ satisfy assumptions (A1)-(A5) with $x_0=0$ and arbitrary $r_0>0$, and  that $\gamma_{x_0}(r)=1$ and $$I_{x_0}(r)=\frac{2K}{\kappa+1}r_0^{\kappa+1}-\frac{2K}{\kappa+1}r^{\kappa+1}.$$
 	We then have $$\Lambda=\int_{r_0}^\infty\E^{\frac{2K}{\kappa+1}u^{\kappa+1}}\int_u^\infty\E^{-\frac{2K}{\kappa+1}v^{\kappa+1}}\D v\,\D u.$$
 	Observe that \begin{align*}\int_u^\infty\E^{-\frac{2K}{\kappa+1}v^{\kappa+1}}\D v&=(2K)^{\frac{-2\kappa-1}{\kappa+1}}(\kappa+1)^{\frac{\kappa}{\kappa+1}}\int_{\frac{2K}{\kappa+1}u^{\kappa+1}}v^{-\frac{\kappa}{\kappa+1}}\E^{-v}\D v\\
 	&=(2K)^{\frac{1}{\kappa}}(\kappa+1)^{-\frac{\kappa+1}{\kappa}}\Gamma\left(\frac{1}{\kappa+1},\frac{2K}{\kappa+1}u^{\kappa+1}\right),\end{align*}
where $\Gamma(p,q)$, $p,q\in\R$, is the incomplete gamma function. According to \cite[display 6.5.32]{Abramowitz-Stegun-1964}, $$\lim_{u\to\infty}\frac{\Gamma\left(\frac{1}{\kappa+1},\frac{2K}{\kappa+1}u^{\kappa+1}\right)}{u^{-\kappa}\E^{-\frac{2K}{\kappa+1}u^{\kappa+1}}}=2^{-\frac{\kappa}{\kappa+1}}(\kappa+1)^{\frac{\kappa}{\kappa+1}},$$
 which proves that $\Lambda<\infty.$ \qed
 	\end{example}
 
In the next example, we consider a one-dimensional diffusion process with a non-incremental drift coefficient.
 
 \begin{example}\label{ex2}
	Let $\sigma(x)=1$,  and let $b(x)$ be locally Lipschitz continuous, such that $b(x)=-Kx|x|^{\kappa-1}(\cos x+\rho)$ for $x\in B^c_1(0)$, where $K>0$, $\kappa>1$, and $\rho>0$. It is straightforward to verify that $b(x)$ and $\sigma(x)$ satisfy assumptions (A1)-(A5) with $x_0=0$ and arbitrary $r_0>0$. Furthermore,  we have   $\gamma_{x_0}(r)=1$ and $$I_{x_0}(r)=\frac{2K\rho}{\kappa+1}r_0^{\kappa+1}-\frac{2K\rho}{\kappa+1}r^{\kappa+1}-2K\int_{r_0}^ru^{\kappa}\cos u\,\D u.$$
	Note that $$\lim_{r\to\infty}\frac{\int_{r_0}^ru^{\kappa}\cos u\,\D u}{r^{k+\varepsilon}}=\lim_{r\to\infty}\frac{r^k\sin r-r_0^k\sin r_0-\kappa\int_{r_0}^ru^{\kappa-1}\sin u\,\D u}{r^{k+\varepsilon}}=0$$ for any $\varepsilon>0$. Thus, the finiteness of $\Lambda$ reduces on showing the finiteness of $$\int_{r_0}^\infty\E^{\frac{2K\rho}{\kappa+1}u^{\kappa+1}}\int_u^\infty\E^{-\frac{2K\rho}{\kappa+1}v^{\kappa+1}}\D v\,\D u,$$ which, as  shown in Example \ref{eq1}, is finite if, and only if, $\kappa>1$.
	\qed
\end{example}

 In the following example, we discuss uniform ergodicity of a class of Langevin tempered diffusion processes.

  \begin{example}\label{ex3}
 	Let $\alpha\in(0,1/d)$, and let $\pi\in \mathcal{C}^2(\R^d)$ be strictly positive, defined as 
 	$\pi(x)\df c|x|^{-\nicefrac{1}{\alpha}}$ for some $c>0$ and for all $x\in B_1^c(0)$. Additionally, assume that
 	 $\int_{\R^d}\pi(x)\D x=1$.
 	For $\beta\in(0,(1+\alpha(2-d))/2)$ and $x\in\R^d$,  define
 	\begin{equation*}
 	\sigma(x)\df(\pi(x))^{-\beta}I_d,
 	\end{equation*}
 	and
 	\begin{equation*}
 	b(x)\df\frac{1}{2}(\sigma(x)\sigma(x)^T\nabla\log(\pi(x))+\nabla\cdot \sigma(x)^T\sigma(x))
 	=\frac{1-2\beta}{2}\pi(x)^{-2\beta}\nabla\log(\pi(x)).
 	\end{equation*}
 	Observe that $$\sigma(x)=c^{-\beta}|x|^{\frac{\beta}{\alpha}}\quad\text{and}\quad b(x)= -\frac{1-2\beta}{2\alpha}c^{-2\beta}x|x|^{\frac{2\beta}{\alpha}-2},\qquad x\in B_1^c(0).$$
 Note  that $b(x)$ and $\sigma(x)$ do not in general satisfy assumptions (A2) and A(3). However, observe that in the proof of Theorem \ref{tm:TV} we  did not actually use the fact that $\{X(t)\}_{t\ge0}$ is a unique strong solution to \eqref{eq1}. All that we required is that the martingale problem for $(b,\sigma\sigma^T)$ is well posed, which is equivalent to  \eqref{eq1} admitting a unique (in distribution) weak solution (see \cite[Theorem V.20.1]{Rogers-Williams-Book-II-2000}). Furthermore,  in \cite[Proposition 15]{Fort-Roberts-2005} it has been shown that the SDE in \eqref{eq1}
 admits a weak solution
 $(\Omega,\mathcal{F},\{\mathcal{F}(t)\}_{t\ge0},\{B(t)\}_{t\ge0},\{X(t)\}_{t\ge0},\Prob)$,
which is a conservative strong Markov process with continuous sample paths.
Moreover, it is open-set irreducible, aperiodic, every compact set is petite, and
$\uppi(\D x)\df\pi(x)\D x$ is its unique invariant probability measure.
Here, $(\Omega,\mathcal{F}, \{\mathcal{F}(t)\}_{t\ge0},\{B(t)\}_{t\ge0},\Prob)$ is a
standard $n$-dimensional Brownian motion.
 Now, take $x_0=0$ and $r_0>1$.	It is straightforward to verify   that $\gamma_{x_0}(r)=c^{-2\beta}r^{\frac{2\beta}{\alpha}}$ and $$I_{x_0}(r)=\frac{1-2\beta}{\alpha}\ln r_0-\frac{1-2\beta}{\alpha}\ln r.$$
 		We then have $$\Lambda=c^{2\beta}\int_{r_0}^\infty\E^{\frac{1-2\beta}{\alpha}\ln u}\int_u^\infty\frac{\E^{-\frac{1-2\beta}{\alpha}\ln v}}{v^{\frac{2\beta}{\alpha}}}\D v\,\D u=c^{2\beta}\int_{r_0}^\infty u^{\frac{1-2\beta}{\alpha}}\int_u^\infty v^{-\frac{1}{\alpha}}\D v\,\D u.$$ This simplifies to $$\Lambda=\frac{\alpha c^{2\beta}}{1-\alpha}\int_{r_0}^\infty u^{1-\frac{2\beta}{\alpha}}\D u,$$ which is finite if, and only if,  $\alpha<\beta.$\qed
 	 \end{example}

The results discussed in this article can be  used to establish uniform ergodicity of a class of Markov processes with jumps. Recall that a subordinator  $\{S(t)\}_{t\ge0}$ is a non-decreasing L\'{e}vy process on $\left[0,\infty\right)$. If $\{X(t)\}_{t\ge0}$ and $\{
 S(t)\}_{t\ge0}$ are independent, 
 then the process $X^{S}(t)\df X(S(t))$, $t\ge0$, obtained by applying random time change to $\{X(t)\}_{t\ge0}$   through $\{S(t)\}_{t\ge0}$, is referred to as the subordinate process $\{X(t)\}_{t\ge0}$ with subordinator $\{S(t)\}_{t\ge0}$ in the sense of Bochner. It is easy to see that $\{X^S(t)\}_{t\ge0}$ is  a Markov process with  the transition kernel
 $$p^S(t,x,\D y)=\int_{\left[0,\infty\right)} p(s,x,\D y)\mu_t(\D s),$$
 where $\upmu_t(\D s)=\mathbb{P}(S(t)\in\D s)$ is the transition probability of $S(t)$, $t\ge0$. Additionally, it is straightforward to check that if $\uppi(\D x)$ is an invariant probability measure for $\{X(t)\}_{t\ge0}$, then it is also invariant for the subordinate process $\{X^S(t)\}_{t\ge0}$.
 Moreover, if $$\lVert p(t,x,\D {y})
 -\uppi(\D {y})\rVert_{{\rm TV}}\le B(x)r(t),\qquad x\in\R^d,\quad t\ge0,$$ for some $B:\R^d\to[0,\infty)$ and measurable $r:[0,\infty)\to[0,\infty)$,
 then
 \begin{align*}
 \lVert p^S(t,x,\D {y})
 -\uppi(\D {y})\rVert_{{\rm TV}}&=\left\lVert \int_{\left[0,\infty\right)} (p(s,x,\D y)-\uppi(\D {y}))\upmu_t(\D s)
 \right\rVert_{{\rm TV}}\\&\le c(x)\int_{\left[0,\infty\right)}r(s)\upmu_t(\D s)\\
 &=B(x)\mathbb{E}[r(S(t))].
 \end{align*}
 Furthermore, since $\mathbb{E}[S(t)]>0$ for $t>0$, the strong law of large numbers (see, e.g., \cite[Theorems 36.5 and 36.6]{Sato-Book-1999}) implies that $\lim_{t\to\infty} S(t)=\infty$, almost surely. Hence, if $r(t)=\E^{-\beta t}$ for some $\beta>0$, the dominated convergence theorem gives that $$\lim_{t\to\infty}\mathbb{E}[r(S(t))]=0.$$ Therefore, if  $\{X(t)\}_{t\ge0}$ is uniformly ergodic, then $\{X^S(t)\}_{t\ge0}$ is also uniformly ergodic for any subordinator $\{S(t)\}_{t\ge0}$.

\section*{Acknowledgement} 
\noindent The author thanks the anonymous referee for the helpful comments that have led to significant
improvements of the results in the article.
This work has been supported by \textit{Croatian Science Foundation} under project 2277.

\section*{Conflict of interest}
 \noindent The author declares that he has no conflict of interest.

\bibliographystyle{abbrv}
\bibliography{References}

\end{document}